\documentclass{birkjour}

\usepackage{graphicx}
\usepackage{amssymb}
\usepackage{amsthm}
\newtheorem{thm}{Theorem}
\newtheorem{cor}{Corollary}
\newtheorem{lem}{Lemma}
\newtheorem{pro}{Proposition}

\theoremstyle{remark}

\begin{document}
\title{Spherical bodies of constant width}

\pagestyle{myheadings} \markboth{Spherical bodies of constant width}{M. Lassak, M. Musielak}

\author{MAREK LASSAK AND MICHA\L \ MUSIELAK}
\address{Institute of Mathematics and Physics\\University of Technology and Life Sciences\\al. Kaliskiego 7\\85-796 Bydgoszcz, Poland}
\email{marek.lassak@utp.edu.pl, michal.musielak@utp.edu.pl}
\subjclass{52A55}
\keywords{Spherical convex body, spherical geometry, hemisphere, lune, width, constant width, thickness, diameter, extreme point.}
\date{} 

\begin{abstract}
The intersection $L$ of two different non-opposite hemispheres $G$ and $H$ of a $d$-dimensional sphere $S^d$ is called a lune. By the thickness of $L$ we mean the distance of the centers of the $(d-1)$-dimensional hemispheres bounding $L$. 
For a hemisphere $G$ supporting a 
convex body $C \subset S^d$ we define ${\rm width}_G(C)$ as 
the thickness of the narrowest lune or lunes of the form $G \cap H$ containing $C$.
If ${\rm width}_G(C) =w$ for every hemisphere $G$ supporting $C$, we say that $C$ is a body of constant width $w$. We present properties of these bodies. In particular, we prove that the diameter of any spherical body $C$ of constant width $w$ on $S^d$ is $w$, and that if $w < \frac{\pi}{2}$, then $C$ is strictly convex. Moreover, we are checking when spherical bodies of constant width and constant diameter coincide.
\end{abstract}

\maketitle
 
\section{Introduction}

Consider the unit sphere $S^d$ in the $(d+1)$-dimensional Euclidean space $E^{d+1}$ for $d\geq 2$. 
The intersection of $S^d$ with any two-dimensional subspace of $E^{d+1}$ is called a {\it great circle} of $S^d$.
By a {\it $(d-1)$-dimensional great sphere} of $S^d$ we mean the common part of 
$S^d$ with any hyper-subspace of $E^{d+1}$.
The $1$-dimensional great spheres of $S^2$ are called {\it great circles}.
By a pair of {\it antipodes} of $S^d$ we understand a pair of points of intersection of $S^d$ with a straight line through the origin of $E^{d+1}$.

Clearly, if two different points $a, b \in S^d$ are not antipodes, there is exactly one great circle containing them.
As the {\it arc} $ab$ connecting $a$ and $b$ we define the shorter part of the great circle containing these points. 
The length of this arc is called the {\it spherical distance $|ab|$ of $a$ and $b$}, or shortly {\it distance}.
Moreover, we agree that the distance of coinciding points is $0$, and that of any pair of antipodes is $\pi$.

A {\it spherical ball $B_\rho(x)$ of radius $\rho \in (0, {\frac{\pi}{2}}]$}, or shorter {\it a ball} is the set of points of $S^d$ at distances at most $\rho$ from a fixed point $x$, which is called the {\it center} of this ball. 
An {\it open ball} (a {\it sphere}) is the set of points of $S^d$ having distance smaller than (respectively, exactly) $\rho$ from a fixed point.
A spherical ball of radius $\frac{\pi}{2}$ is called a {\it hemisphere}. 
So it is the common part of $S^d$ and a closed half-space of $E^{d+1}$.
We denote by $H(m)$ the hemisphere with center $m$.
Two hemispheres with centers at a pair of antipodes are called {\it opposite}.

A {\it spherical $(d-1)$-dimensional ball of radius $\rho \in (0, {\frac{\pi}{2}}]$} is the set of points of a $(d-1)$-dimensional great sphere of $S^d$ which are at distances at most $\rho$ from a fixed point.
We call it the {\it center} of this ball.  
The $(d-1)$-dimensional balls of radius $\frac{\pi}{2}$ are called {\it $(d-1)$-dimensional hemispheres}, and {\it semicircles} for $d=2$.

A set $C \subset S^d$ is said to be {\it convex} if no pair of antipodes belongs to $C$ and if for every $a, b \in C$ we have $ab \subset C$.  
A closed convex set on $S^d$ with non-empty interior is called a {\it convex body}. 
Some  basic references on convex bodies and their properties are \cite{Ha}, \cite{Le} and \cite{Sa}.
A short survey of other definitions of convexity on $S^d$ is given in Section 9.1 of \cite{DGK}.

Since the intersection of every family of convex sets is also convex, for every set $A \subset S^d$ contained in an open hemisphere of $S^d$ there is the smallest convex set ${\rm conv} (A)$ containing $Q$. 
We call it {\it the convex hull of} $A$.

Let $C \subset S^d$ be a convex body. 
Let $Q \subset S^d$ be a convex body or a hemisphere.  
We say that $C$ {\it touches $Q$ from inside} if $C \subset Q$ and ${\rm bd} (C) \cap {\rm bd} (Q) \not = \emptyset$. 
We say that $C$ {\it touches $Q$ from outside} if $C \cap Q \not = \emptyset$ and ${\rm int} (C) \cap {\rm int} (Q) = \emptyset$. 
In both cases, points of ${\rm bd} (C) \cap {\rm bd} (Q)$ are called {\it points of touching}.
In the first case, if $Q$ is a hemisphere, we also say that $Q$ {\it supports} $C$, or {\it supports $C$ at $t$}, provided $t$ is a point of touching.
If at every boundary point of $C$ exactly one hemisphere supports $C$, we say that $C$ is {\it smooth}.

If hemispheres $G$ and $H$ of $S^d$ are different and not opposite, then $L = G \cap H$ is called {\it a lune} of $S^d$. 
This notion is considered in many books and papers (for instance, see \cite{VB}). 
The $(d-1)$-dimensional hemispheres bounding $L$ and contained in $G$ and $H$, respectively, are denoted by $G/H$  and $H/G$.

Observe that $(G/H) \cup (H/G)$ is the boundary of the lune $G \cap H$.  
Denote by $c_{G/H}$ and $c_{H/G}$ the centers of $G/H$ and $H/G$, respectively.
By {\it corners} of the lune $G \cap H$ we mean points of the set $(G/H) \cap (H/G)$. 
In particular, every lune on $S^2$ has two corners. 
They are antipodes. 

We define the {\it thickness $\Delta (L)$ of a lune} $L = G \cap H$ on $S^d$ as the spherical distance of the centers of the $(d-1)$-dimensional hemispheres $G/H$ and $H/G$ bounding $L$.
Clearly, it is equal to each of the non-oriented angles $\angle c_{G/H}rc_{H/G}$, where $r$ is any corner of $L$. 

Compactness arguments show that for any hemisphere $K$ supporting a convex body $C \subset S^d$ there is at least one hemisphere $K^*$ supporting $C$ such that the lune $K \cap K^*$ is of the minimum thickness.
In other words, there is a ``narrowest" lune of the form $K \cap K'$ over all hemispheres $K'$ supporting $C$. 
The thickness of the lune $K \cap K^*$ is called {\it the width of $C$ determined by $K$.} 
We denote it by ${\rm width}_K (C)$. 

We define the {\it thickness} $\Delta (C)$ of a spherical convex body $C$ as the smallest width of $C$. 
This definition is analogous to the classical definition of thickness (called also minimal width) of a convex body in Euclidean space.

By {\it the relative interior} of a convex set $C \subset S^d$ we mean the interior of $C$ with respect to the smallest sphere $S^k \subset S^d$ that contains $C$.

\vskip0.5cm 
\section{A few lemmas on spherical convex bodies}

\begin{lem}\label{intersection}
Let $A$ be a closed set contained in an open hemisphere of $S^d$. Then ${\rm conv} (A)$ coincides with the intersection of all hemispheres containing $A$.
\end{lem}

\begin{proof}
First, let us show that ${\rm conv} (A)$ is contained in the intersection of all hemispheres containing $A$.
Take any hemisphere $H$ containing $A$ and denote by $J$ the open hemisphere from the formulation of our lemma.
Recall that $A\subset J$ and $A\subset H$.
Thus since $J\cap H$ is a convex set,
we obtain ${\rm conv} (A) \subset {\rm conv} (J\cap H) = J\cap H\subset H$.
Thus, since ${\rm conv} (A)$ is contained in any hemisphere 
that contains $A$, also ${\rm conv} (A)$ is a subset of the intersection of all those hemispheres.

Now we intend to show that the intersection of all hemispheres containing $A$ is contained in ${\rm conv} (A)$.
Assume the opposite, i.e., that there is a point  $x \notin {\rm conv} (A)$ which belongs to every hemisphere containing $A$. 
Since $A$ is closed, by Lemma 1 of \cite{L2} the set ${\rm conv} (A)$ is also closed. 
Hence there is an $\varepsilon >0$ such that $B_\varepsilon (x) \cap {\rm conv} (A) = \emptyset$.
Since these two sets are convex, we may apply the following more general version of Lemma 2 of \cite{L2}: {\it any two convex disjoint sets on $S^d$ are subsets of two opposite hemispheres} (which is true again by the separation theorem for convex cones in $E^{d+1}$). 
So $B_\varepsilon (x)$ and ${\rm conv} (A)$ are in some two opposite hemispheres.
Hence $x$ does not belong to this of them which contains ${\rm conv} (A)$. 
Clearly, that one contains also $A$.
This contradicts our assumption on the choice of $x$, and thus the proof is finished.
\end{proof}

We omit a simple proof of the next lemma, which is analogous to the situation in $E^d$ and needed a few times later.
Here our hemisphere plays the role of a closed half-space there.

\begin{lem}\label{support} Let $C$ be a spherical convex body. 
Assume that a hemisphere $H$ supports $C$ at a point $p$ of the relative interior of a convex set $T\subset C$. 
Then $T \subset {\rm bd} (H)$.\end{lem}

\begin{lem}\label{distance} 
Let $K, M$ be hemispheres such that the lune $K\cap M$ is of thickness smaller than $\frac{\pi}{2}$.
 Denote by $b$ the center of $M/K$.
Every point of $K \cap M$ at distance $\frac{\pi}{2}$ from $b$ is a corner of $K \cap M$.
\end{lem}

\begin{proof}
Denote the center of $K/M$ by $a$. 
Take any point $p\in K\cap M$.
Let us show that there are points $x \in (K/M) \cap( M/K)$ and $y \in ab$ such that $p \in xy$. 
 
If $p=b$ then it is obvious.
Otherwise there is a unique point $q \in K/M$ such that $p \in bq$. 
Moreover, there exists $x \in (K/M) \cap( M/K)$ such that $q \in ax$. 
The reader can easily show that points $p, q$ belong to the triangle $abx$ and thus observe that there exists $y \in ab$ such that $p \in xy$, which confirms the statement from the first paragraph of the proof.

We have $|by|\le |ba|< \frac{\pi}{2}$. 
The inequality $|by| < \frac{\pi}{2}$ means that $y$ is in the interior of $H(b)$. 
Of course, $|bx| = \frac{\pi}{2}$, which means that $x \in {\rm bd} (H(b))$. 
From the two preceding sentences we conclude that $xy$ is a subset of $H(b)$ with $x$ being its only point on ${\rm bd} (H(b))$.
Thus, if $|pb|=\frac{\pi}{2}$, we conclude that $p \in {\rm bd} (H(b))$, and consequently $p=x$ which implies that $p$ is a corner of $K\cap M$.
The last sentence means that the thesis of our lemma holds true.
\end{proof}

\begin{lem}\label{convexhull}
Let $o \in S^d$ and $0 < \mu < \frac{\pi}{2}$.
For every $x \in S^d$ at distance $\frac{\pi}{2}$ from $o$ denote by $x'$ the point of the arc $ox$ at distance $\mu$ from $x$. 
Consider two points $x_1,x_2$ at distance $\frac{\pi}{2}$ from $o$ such that $|x_1x_2| < \pi - \mu$. Then for every $x \in x_1x_2$ we have 

$$B_\mu (x')  \subset {\rm conv} (B_\mu(x_1') \cup B_\mu(x_2')).$$
\end{lem}

\begin{proof}
Let $o,m$ be points of $S^d$ and $\rho$ be a positive number less than $\frac{\pi}{2}$.
Let us show that

$$B_\rho (o)\subset H(m) \ \ {\rm if \ and \ only \ if} \ \ |om|\le \frac{\pi}{2}- \rho. \eqno (1)$$

First assume that $B_\rho (o)\subset H(m)$. 
Let $b$ be the boundary point of $B_\rho (o)$ such that $o \in mb$.
We have: $|om|=|bm|-|ob| = |bm| - \rho \le \frac{\pi}{2}- \rho$, which confirms the ``only if" part of (*).
Assume now that $|om|\le \frac{\pi}{2}-\rho$.
Let $b$ be any point of $B_\rho (o)$.
We have: $|bm|\le |bo|+|om|\le \rho + \left( \frac{\pi}{2}- \rho \right) = \frac{\pi}{2}$.
Therefore every point of $B_\rho (o)$ is at a distance at most $\frac{\pi}{2}$ from $m$.
Hence $B\subset H(m)$, which confirms the ``if" part of (1).
So (1) is shown.

Lemma 1 of \cite{L2} guarantees that $Y={\rm conv} (B_\mu(x_1') \cup B_\mu(x_2'))$ is a closed set as convex hull of a closed set. 
Consequently, from Lemma \ref{intersection} we see that $Y$ is the intersection of all hemispheres containing $Y$.
Moreover, observe that an arbitrary hemisphere contains a set if and only if it contains the convex hull of it.
Hence $Y$ is the intersection of all hemispheres containing $B_\mu(x_1') \cup B_\mu(x_2')$.

As a result of the preceding paragraph, in order to prove the statement of our lemma 
it is sufficient to show that every hemisphere $H(m)$ containing $B_\mu(x_1') \cup B_\mu(x_2')$ contains also $B_\mu (x')$.
Thus, having (1) in mind we see that in order to verify this it is sufficient to show that for any $m \in S^d$ 

$$|x_1'm|\le \frac{\pi}{2}-\mu \ \ {\rm and} \ \ |x_2'm|\le \frac{\pi}{2}-\mu \ \ {\rm imply} \ \ |x'm|\le \frac{\pi}{2}- \mu. \eqno (2)$$

Let us assume the first two of these inequalities and show the third one.

Observe that $x,x_1'$ and $x_2'$ belong to the spherical triangle $x_1x_2o$.
Therefore the arcs $xo$ and $x_1'x_2'$ intersect.
Denote the point of intersection by $g$. 

In this paragraph we consider the intersection of $S^d$ with the three-dimensional subspace of $E^{d+1}$ containing $x_1', x_2', m$. 
Observe that this intersection is a two-dimensional sphere concentric with $S^d$. 
Denote this sphere by $S^2$. 
Denote by $\overline{o}$ the other unique point on $S^2$ such that the triangles $x_1'x_2'o$ and $x_1'x_2'\overline{o}$ are congruent.
By the first two inequalities of (2) we obtain $m\in B_{\frac{\pi}{2}- \mu}(x_1') \cap B_{\frac{\pi}{2}- \mu}(x_2')$.
Observe that $g\overline{o} \cup go$ dissects $B_{\frac{\pi}{2}- \mu}(x_1') \cap B_{\frac{\pi}{2}- \mu}(x_2')$ into two parts such that $x_1'$ belongs to one of them and $x_2'$ belongs to the other.
Therefore at least one of the arcs $x_1'm$ and $x_2'm$, say $x_1'm$ intersects $g\overline{o}$ or $go$, say $go$.
Denote this point of the intersection by $s$.
Taking the first assumption of (2) into account and using two times the triangle inequality we obtain $|og |= \left( |os|+|x_1's|\right) - |x_1's| + |sg|\ge |ox_1'| - |x_1's| + |sg|= \frac{\pi}{2}- \mu  - |x_1's| + |sg|\ge |x_1'm|  - |x_1's| + |sg|=|sm|+|sg|\ge |gm|$.

Applying the just obtained inequality and looking now again on the whole $S^d$ we get $|x'm|\le |x'g|+|gm|\le |x'g|+ |og| = |x'o| = \frac{\pi}{2}- \mu$ which is the required inequality in (2). 
Thus by (2) also our lemma holds true.
\end{proof}

\begin{lem}\label{extreme} 
Let $C\subset S^d$ be a convex body. Every point of $C$ belongs to the convex hull of at most $d+1$ extreme points of $C$.
\end{lem}

\begin{proof}
We apply induction with respect to $d$.
For $d=1$ the thesis is trivial since every convex body on $S^1$ is a spherical arc. 
Let $d \ge 2$ be a fixed integer.
Assume that for each $k=1,2, \dots , d-1$ every boundary point of a spherical convex body $\widehat{C}\subset S^k$ belongs to convex hull of at most $k+1$  extreme points of $\widehat{C}$. 

Let $x$ be a point of $C$. 
Take an extreme point $e$ of $C$. 
If $x$ is not a boundary point of $C$, take the boundary point $f$ of $C$ such that $x\in ef$.
In the opposite case put $f=x$. 

If $f$ is an extreme point of $C$, the thesis follows immediately.
In the opposite case take a hemisphere $K$ supporting $C$ at $f$. 
Put $C'=\textrm{bd} (K)\cap C$.
Observe that every extreme point of $C'$ is also an extreme point of $C$.
Let $Q$ be the intersection of the smallest linear subspace of $E^{d+1}$ containing $C'$ with $S^d$. 
Clearly, $Q$ is isomorphic to $S^k$ for a $k<d$.
Moreover, $C'$ has non-empty relative interior with respect to $Q$, because otherwise there would exist a smaller linear subspace of $E^{d+1}$ containing $C'$. 
Thus, by the inductive assumption we obtain that $f$ is in the convex hull of a set $F$ of at most $d$ extreme points of $C$.
Therefore $x\in \textrm{conv} (\{e\} \cup F)$ which means that $x$ belongs to the convex hull of $d+1$ extreme points of $C$. 
This finishes the inductive proof.
\end{proof}

The proof of the following $d$-dimensional lemma is analogous to the two-dimensional Lemma 4.1 from \cite{LaMu} shown there for wider class of reduced spherical convex bodies. 

\begin{lem}\label{l_smooth} 
 Let $C\subset S^d$ be a spherical convex body with $\Delta (C) > \frac{\pi}{2}$ and let $L \supset C$ be a lune such that $\Delta (L) = \Delta (C)$.
Each of the centers of the $(d-1)$-dimensional hemispheres bounding $L$ belongs to the boundary of $C$ and both are smooth points of the boundary of $C$.
\end{lem} 

Having in mind the next lemma, we note the obvious fact that the diameter of a convex body $C \subset S^d$ is realized only for some pairs of points of ${\rm bd}(C)$.

\begin{lem}\label{diameter}
Assume that the diameter of a convex body $C \subset S^d$ is realized for points $p$ and $q$. 
The hemisphere $K$ orthogonal to $pq$ at $p$ and containing $q \in K$ supports $C$.
\end{lem}

\begin{proof}
Denote the diameter of $C$ by $\delta$.

Assume first that $\delta > \frac{\pi}{2}$.
The set of points at distance at least $\delta$ from $q$ is the ball $B_{\pi - \delta}(q')$, where $q'$ is the antipode of $q$. 
Clearly, $K$ has only $p$ in common with $B_{\pi - \delta}(q')$.

Since the diameter $\delta$ of $C$ is realized for $pq$, every point of $C$ is at distance at most $\delta$ from $q$.
Thus $C$ has empty intersection with the interior of $B_{\pi - \delta}(q')$.

Assume that $K$ does not contain $C$.
Then $C$ contains a point $b \not \in K$.
Observe that the arc $bp$ has nonempty intersection with the interior of $B_{\pi - \delta}(q')$  [the reason: $K$ is the only hemisphere touching $B_{\pi - \delta}(q')$ from outside at $p$]. 
On the other hand, by the convexity of $C$ we have $bp \subset C$.
This contradicts the fact from the preceding paragraph that 
$C$ has empty intersection with the interior of $B_{\pi - \delta}(q')$.
Consequently, $K$ contains $C$.

Now consider the case when $\delta \leq \frac{\pi}{2}$.
For every $y \not \in K$ we have $|pq| < |yq|$ which by $|pq| = \delta$ implies $y \not \in C$. Thus always if $y \in C$, then $y \in K$.
\end{proof}

\smallskip
Let us apply our Lemma \ref{diameter} for a convex body $C$ of diameter larger than $\frac{\pi}{2}$. 
Having in mind that the center $k$ of $K$ is in $pq$ and thus in $C$, by Part III of Theorem 1 in \cite{L2} we obtain that $\Delta (K \cap K^*) > \frac{\pi}{2}$. 
This gives the following corollary which implies the next one.
The symbol ${\rm diam} (C)$ denotes the diameter of $C$.

\begin{cor} \label{ortho}
Let $C \subset S^d$ be a convex body of diameter larger than $\frac{\pi}{2}$ and let 
${\rm diam} (C)$ be realized for points $p, q \in C$.
Take the hemisphere $K$ orthogonal to $pq$ at $p$ which supports $C$.
Then ${\rm width}_K (C) > \frac{\pi}{2}$. 
\end{cor}

\begin{cor} \label{family}
Let $C \subset S^d$ be a convex body of diameter larger than $\frac{\pi}{2}$ and let $\mathcal K$ denote the family of all hemispheres supporting $C$.  
Then $\max_{K \in \mathcal K} {\rm width}_K (C) > \frac{\pi}{2}$. 
\end{cor}

\vskip0.5cm
\section{Spherical bodies of constant width}

If for every hemisphere supporting a convex body $W \subset S^d$ the width of $W$ determined by $K$ is the same, we say that $W$ is a {\it body of constant width} (see \cite{L2} and for an application also \cite{HN}). 
In particular, spherical balls of radius smaller than $\frac{\pi}{2}$ are bodies of constant width. 
Also every spherical Reuleaux odd-gon (for the definition see \cite{L2}, p. 557) is a convex body of constant width.
Each of the $2^{d+1}$ parts of $S^d$ dissected by $d+1$ pairwise orthogonal $(d-1)$-dimensional spheres of $S^d$ is a spherical body of constant width $\frac{\pi}{2}$, which easily follows from the definition of a body of constant width.   
The class of spherical bodies of constant width is a subclass of the class of spherical reduced bodies considered in \cite{L2} and \cite{LaMu}, and mentioned by \cite{GJPW} in a larger context, (recall that a convex body $R \subset S^d$ is called {\it reduced} if $\Delta (Z) < \Delta (R)$ for every body $Z \subset R$ different from $R$, see also \cite{LM} for this notion in $E^d$). 

By the definition of width and by Claim 2 of \cite{L2}, if $W\subset S^d$ is a body of constant width, then every supporting hemisphere $G$ of $W$ determines a supporting hemisphere $H$ of $W$ for which $G\cap H$ is a lune of thickness $\Delta (W)$ such that the centers of $G/H$ and $H/G$ belong to the boundary of $W$. 
Hence
{\it every spherical body $W$ of constant width is an intersection of lunes of thickness $\Delta (W)$ such that the centers of the $(d-1)$-dimensional hemispheres bounding these lunes belong to ${\rm bd} (W)$}. 
Recall the related question from p. 563 of \cite{L2} if a convex body $W \subset S^d$ is of constant width provided every supporting hemisphere $G$ of $W$ determines at least one hemisphere $H$ supporting $W$ such that $G \cap H$ is a lune with the centers of $G/H$ and $H/G$ in ${\rm bd} (W)$.

Here is an example of a of spherical body of constant width on $S^3$.

\smallskip
\emph{Example}.
Take a circle $X \subset S^3$ (i.e., a set congruent to a circle in $E^2$) of a positive diameter $\kappa < \frac{\pi}{2}$, and a point $y \in S^3$ at distance $\kappa$ from every point $x \in X$. 
Prolong every spherical arc $yx$ by a distance $\sigma \leq {\pi \over 2} - \kappa$ up to points $a$ and $b$ so that $a, y, x, b$ are on one great circle in this order.
All these points $a$ form a circle $A$, and all points $b$ form a circle $B$. 
On the sphere on $S^3$ of radius $\sigma$ whose center is $y$ take the ``smaller" part $A^+$ bounded by the circle $A$. 
On the sphere on $S^3$ of radius $\kappa + \sigma$ with center $y$ take the ``smaller" part $B^+$ bounded by $B$.
For every $x \in X$ denote by $x'$ the point on $X$ such that $|xx'| = \kappa$.
Prolong every $xx'$ up to points $d, d'$ so that $d, x, x', d'$ are in this order and $|dx|= \sigma =|x'd'|$.
For every $x$ provide the shorter piece $C_x$ of the circle with center $x$ and radius $\sigma$ connecting the $b$ and $d$ determined by $x$ and also the shorter piece $D_x$ of the circle with center $x$ of radius $\kappa +\sigma$ connecting the $a$ and $d'$ determined by $x$. 
Denote by $W$ the convex hull of the union of $A^+$, $B^+$ and all pieces $C_x$ and $D_x$. 
It is a body of constant width $\kappa + 2\sigma$ (hint: for every hemisphere $H$ supporting $W$ and every $H^*$ the centers of $H/H^*$ and $H^*/H$ belong to ${\rm bd} (W)$ and the arc connecting them passes through one of our points $x$, or through the point $y$).

\begin{thm}\label{touching ball}  At every boundary point $p$ of a body $W \subset S^d$ of constant width $w > \pi/2$ we can inscribe a unique ball $B_{w- \pi/2}(p')$ touching $W$ from inside at $p$. 
What is more, $p'$ belongs to the arc connecting $p$ with the center of the unique hemisphere supporting $W$ at $p$, and $|pp'|=w-\frac{\pi}{2}$.
\end{thm}

\begin{proof} 
Observe that if a ball touches $W$ at $p$ from inside, then there exists a unique hemisphere supporting $W$ at $p$ such that our ball touches this hemisphere at $p$.
So for any $\rho \in (0, \frac{\pi}{2})$ there is at most one ball of radius $\rho$ touching $W$ from inside at $p$.
Our aim is to show that always we can find one.

In the first part of the proof consider the case when $p$ is an extreme point of $W$.
By Theorem 4 of \cite{L2} there is a lune $L = K \cap M$ of thickness $w$ containing $W$ such that $p$ is the center of $K/M$. 
Denote by $m$ the center of $M$ and by $k$ the center of $K$. 
Clearly, $m\in pk$ and $|pm| = w- \frac{\pi}{2}$.
Since $\textrm{width}_M(W)=w$, by the third part of Theorem 1 of \cite{L2} the ball $B_{w- \pi/2}(m)$ touches $W$ from inside.
Moreover, it touches $W$ from inside at the center of $M^*/M$.
Since $K$ is one of these hemispheres $M^*$, our ball touches $W$ at $p$.
 
In the second part consider the case when $p$ is not an extreme point of $W$. 
From Lemma \ref{extreme} we see that $p$ belongs to the convex hull of a finite set $E$ of extreme points of $W$.   
We do not lose the generality assuming that $E$ is a minimum set of extreme points of $W$ with this property. 
Hence $p$ belongs to the relative interior of ${\rm conv} (E)$. 

Take a hemisphere $K$ supporting $W$ at $p$ and denote by $o$ the center of $K$.
Since $p$ belongs to the relative interior of ${\rm conv} (E)$, by Lemma \ref{support} we obtain ${\rm conv} (E) \subset {\rm bd} (K)$.
Moreover, ${\rm conv} (E)$ is a subset of the boundary of $W$.

We intend to show that for every $x \in {\rm conv} (E)$ the inclusion

$$B_{w- \frac{\pi}{2}}(x') \subset W \eqno (3)$$

\noindent
holds true, where $x'$ denotes the point on $ox$ at distance $w - \frac{\pi}{2}$ from $x$.

By Lemma \ref{convexhull} for  $w = \mu$, if (3) holds true for $x_1,x_2 \in {\rm conv} (E)$, then (3) is also true for every point of the arc $x_1x_2$. 
Applying this lemma a finite number of times and having in mind the first part of this proof, we conclude that (3) is true for every point of ${\rm conv} (E)$, so in particular for $p$. 
Clearly, the ball $B_{w- {\pi \over 2}}(p')$ supports $W$ at $p$ from inside. 

Both parts of the proof confirm the thesis of our theorem.
\end{proof}

By Lemma \ref{l_smooth} we obtain the following proposition generalizing Proposition 4.2 from  \cite{LaMu} for arbitrary dimension $d$. 
We omit an analogous proof.
  
\begin{pro}\label{smooth} 
Every spherical body of constant width larger than $\frac{\pi}{2}$ (and more general, every reduced body of thickness larger than $\frac{\pi}{2}$) of $S^d$ is smooth. 
\end{pro}

From Corollary \ref{family} we obtain the following corollary which implies two next ones.

\begin{cor} \label{diamover}
If ${\rm diam} (W) >  \frac{\pi}{2}$ for a body of $W \subset S^d$ of constant width $w$, then $w >  \frac{\pi}{2}$.
\end{cor}

\begin{cor} \label{diamless}
For every body of $W \subset S^d$ of constant width $w \leq \frac{\pi}{2}$ we have ${\rm diam} (W) \leq  \frac{\pi}{2}$.
\end{cor}

\begin{cor} \label{H(p)}
{\it Let $p$ be a point of a body $W \subset S^d$ of constant width at most $\frac{\pi}{2}$. Then $W \subset H(p)$.}
\end{cor}

The following theorem generalizes Theorem 5.2 of \cite{LaMu} proved there for $d=2$ only.

\begin{thm} \label{strictly} 
Every spherical convex body of constant width smaller than $\frac{\pi}{2}$ on $S^d$ is strictly convex. 
\end{thm}

\begin{proof}
Take a body $W$ of constant width $w < \frac{\pi}{2}$ and assume it is not strictly convex. 
Then there is a supporting hemisphere $K$ of $W$ that supports $W$ at more than one point.
By Claim 2 of \cite{L2} the centers $a$ of $K/K^*$ and $b$ of $K^*/K$ belong to ${\rm bd} (W)$.
Since $K$ supports $W$ at more than one point, $K/K^*$ contains also a boundary point $x \not = a$ of $W$.
By the first statement of Lemma 3 of \cite{L2} we have $|xb| > |ab| = w$.
Hence ${\rm diam} (W) > w$. 

By Corollary \ref{diamless} we have ${\rm diam} (W) \leq \frac{\pi}{2}$.
By Theorem 3 of \cite{L2} we see that $w = {\rm diam} (W)$.
This contradicts the inequality  ${\rm diam} (W) > w$ from the preceding paragraph.
The contradiction means that our assumption that $W$ is not strictly convex must be false. 
Consequently, $W$ is strictly convex.
\end{proof}
 
On p. 566 of \cite{L2} the question is put if for every reduced spherical body $R \subset S^d$ and for every $p \in {\rm bd} (R)$ there exists a lune $L \supset R$ fulfilling $\Delta (L) = \Delta (R)$ with $p$ as the center of one of the two $(d-1)$-dimensional hemispheres bounding this lune. 
The following theorem gives the positive answer in the case of spherical bodies of constant width. 
It is a generalization of the version for $S^2$ given as Theorem 5.3 in \cite{LaMu}. The idea of the proof of our theorem below for $S^d$ substantially differs from the mentioned one for $S^2$.

\begin{thm} \label{center} 
For every body $W \subset S^d$ of constant width $w$ and every $p \in {\rm bd} (W)$ there exists a lune $L \supset W$ fulfilling $\Delta (L) = w$ with $p$ as the center of one of the two $(d-1)$-dimensional hemispheres bounding this lune.
\end{thm}

\begin{proof}
Part I for $w < \frac{\pi}{2}$. 

By Theorem \ref{strictly} the body $W$ is strictly convex, which means that every its boundary point is extreme. 
Thus the thesis follows from Theorem 4 of \cite{L2}. 

\smallskip
Part II for $w = \frac{\pi}{2}$. 

If $p$ is an extreme point of $W$ we again apply Theorem 4 of \cite{L2}.

Consider the case when $p$ is not an extreme point. 
Take a hemisphere $G$ supporting $W$ at $p$.
Applying Corollary \ref{H(p)} we see that $W \subset H(p)$.
Clearly, the lune $H(p) \cap G$ contains $W$.
The point $p$ is at distance $\frac{\pi}{2}$ from every corner of this lune and also from every point of the opposite $(d-1)$-dimensional hemisphere bounding the lune. 
Hence this is a lune that we are looking for.

\smallskip
Part III, for $w > \frac{\pi}{2}$. 

By Lemma \ref{extreme} the point $p$ belongs to the convex hull ${\rm conv} (E)$ of a finite set $E$ of extreme points of $W$.
 We do not lose the generality assuming that $E$ is a minimum set of extreme points of $W$ with this property. 
Hence $p$ belongs to the relative interior of ${\rm conv} (E)$. 
By Proposition \ref{smooth} we know that there is a unique hemisphere $K$ supporting $W$ at $p$.
Since $p$ belongs to the relative interior of ${\rm conv} (E)$, by Lemma \ref{support} we have ${\rm conv} (E)\subset {\rm bd} (K)$.
Moreover, ${\rm conv} (E)$ is a subset of the boundary of $W$.

By Theorem 4 of \cite{L2} for every $e \in E$ there exists a hemisphere $K_e^*$ (it plays the part of $K^*$ in Theorem 1 of \cite{L2}) supporting $W$ such that the lune $K \cap K_e^*$ is of thickness $\Delta (W)$ with $e$ as the center of $K/K_e^*$.
By Proposition \ref{smooth}, for every $e$ the hemisphere $K_e^*$ is unique.
For every $e \in E$ denote by $t_e$ the center of $K_e^*/K$ and by $k_e$ the boundary point of $K$ such that $t_e \in ok_e$, where $o$ is the center of $K$. 
So $e, k_e$ are antipodes. 
Denote the set of all these points $k_e$ by $Q$. 

Clearly, the ball $B = B_{\Delta (W) - \frac{\pi}{2}}(o)$ (as in Part III of Theorem 1 in \cite{L2}) touches $W$ from inside at every point $t_e$. 
Moreover, from the proof of Theorem 1 of \cite{L2} and from the earlier established fact that every $e \in E$ is the center of $K/K_e^*$ and every $t_e$ is the center of $K_e^*/K$ we obtain that $o$ belongs to all the arcs of the form $et_e$. 

Put $U= {\rm conv}(Q \cup \{o\})$. 
Denote by $U_B$ the intersection of $U$ with the boundary of $B$, and by $U_W$ the intersection of $U$ with the boundary of $W$.
Having in mind this construction we see the following one-to-one correspondence between some pairs of points in $U_B$ and $U_W$. 
Namely, between the pairs of points of $U_B$ and $U_W$ such that each pair is on the arc connecting $o$ with a point of ${\rm conv} (Q)$.

Now, we will show that $U_W = U_B$. 
Assume the opposite. 
By the preceding paragraph, our opposite assumption means that there is a point $x$ which belongs to $U_W$ but not to $U_B$. 
Hence $|xo|>\Delta (W) - \frac{\pi}{2}$. 
Moreover, there is a boundary point $y$ of the $(d-1)$-dimensional great sphere bounding $K$ such that $o\in xy$ and a point $y' \in oy$ at distance $\Delta (W) -\frac{\pi}{2}$ from $y$. 

 We have $|xy'|=|xo|+|oy|- |yy'|> \left( \Delta (W) - \frac{\pi}{2}\right) +\frac{\pi}{2} - \left( \Delta (W) - \frac{\pi}{2}\right)= \frac{\pi}{2}$.  

By Lemma \ref{extreme} point $x$ belongs the convex hull of a finite set of extreme points of $W$.
Assume for a while that all these extreme points are at distance at most $\frac{\pi}{2}$ from $y'$. 
Therefore all of them are contained in $H(y')$. 
Thus their convex hull is contained in $H(y')$ and so $x\in H(y')$. 
This contradicts the fact established at preceding paragraph that $|xy'|>\frac{\pi}{2}$. 
The contradiction shows that at least one of these extreme points is at distance larger than $\frac{\pi}{2}$ from $y'$.
Take such a point $z$ for which  $|zy'| > \frac{\pi}{2}$.

Since $z$ is an extreme point of $W$, by Theorem 4 of \cite{L2} it is the center of one of the $(d-1)$-dimensional hemispheres bounding a lune $L$ of thickness $\Delta (W)$ which contains $W$. 
Hence by the third part of Lemma 3 of \cite{L2} every point of $L$ 
different from the center of the other $(d-1)$-dimensional hemisphere bounding $L$ is at distance smaller than $\Delta (W)$ from $z$. 
Taking into account, that the distance of these centers is $\Delta (W)$ we see that
the distance of every point of $L$, and in particular of $W$, from $z$ is at most $\Delta (W)$. 
 
By Theorem \ref{touching ball} the ball $B_{\Delta(W) - \frac{\pi}{2}}(y')$ touches $W$ from inside at $y$. 

For the boundary point $v$ of this ball such that $y'\in zv$ we have $|zv| = |zy'| + |y'v| > \frac{\pi}{2} + \left( \Delta (W) - \frac{\pi}{2}\right) = \Delta (W)$, which by $v\in W$ contradicts the result of the preceding paragraph. 
Consequently, $U_W = U_B$. 

Since $U_W=U_B$, the ball $B$ touches $W$ from inside at every point of $U_B$, in particular at the point $t_p$ such that $o\in pt_p$. 
Therefore by Part III of Theorem 1 in \cite{L2} there exists a hemisphere $K_p^*$ supporting $W$ at $t_p$, such that $t_p$ is the center of $K_p^*/K$, $p$ is the center of $K/K_p^*$ and the lune $L=K\cap K_p^*$ is of thickness $\Delta (W)$. 
Consequently, $L$ is a lune announced in our theorem.
\end{proof}

If the body $W$ from Theorem \ref{center} is of constant width greater than $\frac{\pi}{2}$, then  by Proposition \ref{smooth} it is smooth. 
Thus at every $p \in {\rm bd} (W)$ there is a unique supporting hemisphere of $W$, and so the lune $L$ from the formulation of this theorem is unique.
If the constant width of $W$ is at most $\frac{\pi}{2}$, there are non-smooth bodies of constant width (e.g., a Reuleaux triangle on $S^2$) and then for non-smooth $p \in {\rm bd} (W)$ we may have more such lunes. 

Our Theorem \ref{center} and Claim 2 in \cite{L2} imply the first sentence of the following corollary.
The second sentence follows from Proposition \ref{smooth} and the last part of Lemma 3 in \cite{L2}.

\begin{cor} \label{pq} 
For every convex body $W \subset S^d$ of constant width $w$ and for every $p \in {\rm bd} (W)$ there exists $q \in {\rm bd} (W)$ such that $|pq| = w$.  
If $w > \frac{\pi}{2}$, this $q$ is unique.
\end{cor} 

\begin{thm} \label{diam=w}
If $W \subset S^d$ is a body of constant width $w$, then ${\rm diam} (W)=w$.
\end{thm}

\begin{proof} 
If ${\rm diam} (W) \le \frac{\pi}{2}$, then the thesis is an immediate consequence of Theorem 3 in \cite{L2}. 

Assume that ${\rm diam} (W) > \frac{\pi}{2}$.
Take an arc $pq$ in $W$ such that $|pq| = {\rm diam} (W)$.
By Corollary \ref{ortho} this hemisphere $K$ orthogonal to $pq$ at $p$ which contains $q$, contains also $W$.
The center of $K$ lies strictly inside $pq$ and thus by Part III of Theorem 1 in \cite{L2} we have $w>\frac{\pi}{2}$

Having in mind Theorem \ref{center}, consider  
a lune $L \supset W$ with $\Delta (L) = \Delta (W)$ such that $p$ is the center of a $(d-1)$-dimensional hemisphere bounding $L$.
Clearly, $q\in W \subset L$.
Since $W$ is of constant width $w > \frac{\pi}{2}$, we have  $\Delta (L) > \frac{\pi}{2}$.

Thus from the last part of Lemma 3 of \cite{L2} it easily follows that the center of the other $(d-1)$-dimensional hemisphere bounding $L$ is a farthest point of $L$ from $p$.
Since it is at the distance $w$ from $p$, we obtain $w \ge |pq| = {\rm diam} (W)$.

On the other hand, by Proposition 1 of \cite{L2} we have $w\le {\rm diam} (W)$. 

As a consequence, ${\rm diam} (W) = w$.
\end{proof}

\section{Constant width and constant diameter}

We say that a convex body $W \subset S^d$ is {\it of constant diameter} $w$ if the following two conditions hold true 
\begin{list}{}{}
\item  {\rm (i)} ${\rm diam} (W) =w$, 
\item  {\rm (ii)} for every boundary point $p$ of $W$ there exists a boundary point $p'$ of $W$ with $|pp'| = w$.
\end{list}
 
This definition is analogous to the Euclidean notion (compare the beginning of Part 7.6 of \cite{YaBo} for the Euclidean plane, and the bottom of p. 53 of \cite{ChGr} also for higher dimensions).
Here is a theorem similar to the 
planar Euclidean version from \cite{YaBo} (see the beginning of Part 7.6).

\medskip

\begin{thm} \label{iff} 
Every spherical convex body $W \subset S^d$ of constant width $w$ is of constant diameter $w$.
Every spherical convex body $W \subset S^d$ of constant diameter $w \ge \frac{\pi}{2}$ is of constant width $w$. 
\end{thm}

\begin{proof}  
For the proof of the first statement of our theorem assume that $W$ is of constant width $w$. 
Theorem \ref{diam=w} implies (i) and Corollary \ref{pq} implies (ii), which means that $W$ is of constant diameter $w$. 

Let us prove the second statement. 
Let $W \subset S^d$ be a spherical body of constant diameter.
We have to show that $W$ is a body of constant width $w$.

Consider an arbitrary hemisphere $K$ supporting $W$.
As an immediate consequence of two facts from \cite{L2}, namely Theorem 3 and Proposition 1, we obtain that

$${{\rm width}}_K (W) \leq w.\eqno (4)$$

Let us show that ${\rm width}_K (W) = w$.

Make the opposite assumption (so that ${\rm width}_K (W) \not = w$) in order to provide an indirect proof of this equality. 
By (4) this opposite assumption implies that ${\rm width}_K (W) < w$. 

Consider two cases.

\smallskip
At first consider the case when $w > \pi/2$. 

Put $w' ={\rm width}_K (W)$.
There exists a hemisphere $M$ supporting $W$ such that $\Delta (K \cap M) = w'$. 
Denote the center of $K/M$ by $a$ and the center of $M/K$ by $b$. 
From Corollary 2 of \cite{L2} we see that $b \in {\rm bd} (W)$.

We have $\frac{\pi}{2} < w'$ since the opposite means $w' \leq \frac{\pi}{2}$ and then every point of the lune $K \cap M$ is at distance at most $\frac{\pi}{2}$ from the center $b$ of $M/K$ (for $w' = \frac{\pi}{2}$ this is clear by $K \cap M \subset H(b)$, and consequently this is also true if $w' < \frac{\pi}{2}$).
Since $b$ is a boundary point of our body $W$ of constant diameter $w > \pi/2$, we get a contradiction to {\rm (ii)}. 

Since $b$ is a boundary point of the body $W$ of constant diameter, by the assumption {\rm (ii)} there exists $b' \in {\rm bd} (W)$ such that $|bb'| = w$. 
By the definition of the thickness of a lune, we have $|ab| = w'$. 
Observe that the last part of Lemma 3 of \cite{L2} implies that $|uc_{H/G}| \leq |c_{G/H}c_{H/G}|$ for every point $u$ of the lune $H \cap G$. 
This observation applies to our lune $K \cap M$ since $\Delta (K \cap M) >  \frac{\pi}{2}$ (i.e, $w' >  \frac{\pi}{2}$).
Hence we obtain $|b'b| \leq |ab|$, which by the two first sentences of this paragraph gives $w \leq w'$.
This contradicts the inequality $w' < w$ resulting from our opposite assumption that ${\rm width}_K (W) \not = w$.

Consequently,  ${\rm width}_K (W) = w$.

\smallskip
Now consider the case when $w = \frac{\pi}{2}$. 

From ${\rm width}_K (W) < w$ (resulting from our opposite assumption) we obtain  ${\rm width}_K (W) < \pi/2$.
Thus $\Delta (K \cap K^*) < \frac{\pi}{2}$.
Denote by $b$ the center of $K^*/K$.
From Corollary 2 of \cite{L2} we see that $b \in {\rm bd} (W)$.

The set $D = (K/K^*)\cap (K^*/K)$ of corner points of $K\cap K^*$
is isomorphic to $S^{d-2}$.
Moreover, $S^k$ contains at most $k+1$ points pairwise distant by $\frac{\pi}{2}$, which follows from the fact (which is easy to show) that {\it every set of at least $k+2$ points pairwise equidistant on $S^k$ must be the set of vertices of a regular simplex inscribed in $S^k$} (still the distances of these vertices are not $\frac{\pi}{2}$). 
Putting $k= d-2$, we see that $D$ contains at most $d-1$ points pairwise distant by $\frac{\pi}{2}$.
 Therefore there exists a set $P_{max}$ of the maximum number (being at most $d-1$) of points of $W \cap D$ pairwise distant by $\frac{\pi}{2}$. 

Put $T={\rm conv} (P_{max} \cup \{b\})$.
Clearly, $T\subset W$, and even more, since moreover $T\subset {\rm bd} (K^*)$ and $W\subset K^*$, we obtain $T\subset {\rm bd} (W)$.
Take a point $x$ from the relative interior of $T$.  
The inclusion $T\subset {\rm bd} (W)$ implies that $x \in {\rm bd} (W)$. 
Hence by {\rm (ii)} there exists $y \in {\rm bd} (W)$ such that $|xy|= \frac{\pi}{2}$. 
By Lemma \ref{support} we have $T \subset {\rm bd} (H(y))$.
By this inclusion and $b \in T$ we obtain $|by|=\frac{\pi}{2}$.
Thus by Lemma \ref{distance} we have $y \in D$.
As a consequence, the set $P_{max} \cup \{y\}$ is a set of points of $W \cap D$
pairwise distant by $\frac{\pi}{2}$.
This set has a greater number of points than the set $P_{max}$.
This contradiction shows that our assumption ${\rm width}_K (W) \not = w $ is wrong. 
So ${\rm width}_K (W) = w$.

\smallskip
In both cases, from the arbitrariness of the hemisphere $K$ supporting our convex body $W$ we get that $W$ is a body of constant width $w$. 
\end{proof}

\noindent
{\bf Problem.} Is every spherical body of constant diameter $w < \frac{\pi}{2}$ a body of constant width $w$?

\end{document}